\documentclass[draft]{amsart}

\usepackage{amssymb}
\usepackage{url}
\usepackage{amscd}

\usepackage[all]{xy}

\def \N{{\mathbb N}}

\def \R{{\mathbb R}}

\def \1{{\mathbb 1}}

\theoremstyle{plain}
\newtheorem{theorem}{Theorem}
\newtheorem{proposition}{Proposition}
\newtheorem{definition}{Definition}
 
\newtheorem{corollary}{Corollary}

\theoremstyle{remark}
\newtheorem{Rem}{Remark}
\newtheorem{example}{Examples}

\title[G\^ateaux-Differentiability of convex functions]{G\^ateaux-Differentiability of convex functions in infinite dimension.}
\author{Mohammed Bachir and Adrien Fabre}
\date{20/02/2018\\ Laboratoire SAMM 4543, Universit\'e Paris 1 Panth\'eon-Sorbonne. France\\ Mohammed.Bachir@univ-paris1.fr}
\begin{document}
\begin{abstract} It is well known that in $\R^n$, G\^ateaux (hence Fr\'echet) differentiability of a convex continuous function  at some point is equivalent to the existence of the partial derivatives at this point. We prove that this result extends naturally to certain infinite dimensional vector spaces, in particular to Banach spaces having a Schauder basis.  
\end{abstract}


\maketitle

\section{Introduction}
Recall that if $E$ is a topological vector space and $E^*$ its topological dual, the subdifferential of a function $f : E \longrightarrow \R$ at some point $x^*\in E$ is the following subset of $E^*$:
$$\partial f(x^*):= \lbrace p\in E^*: \langle p, x-x^*\rangle \leq f(x)-f(x^*); \hspace{2mm} \forall x\in E \rbrace.$$
Let $U$ be an open subset of $E$ and $f : E\longrightarrow \R$ a function. We say that $f$ is differentiable at $x^*\in U$ in the direction of $h\in E$ if the following limit exists 
$$f'(x^* ; h):=\lim _{t\to 0 \atop t\neq 0}\,{\frac {1}{t}}{\Big (}f (x^*+t h)-f (x^*){\Big )}.$$
We say that $f$ is G\^ateaux-differentiable at $x^*\in U$, if there exists $F\in E^*$ (called the G\^ateaux-derivative of $f$ at $x^*$ and generally denoted by $d f(x^*)$) such that $f'(x^* ; h)=F(h),$
for all $h\in E$.

\vskip5mm

It is well known that if $E$ is a Hausdorff locally convex topological vector space and $f$ is  a convex continuous function then, $f$ is G\^ateaux-differentiable at $x^*\in E$ if and only if, $\partial f(x^*)$ is a singleton (see \cite[Corollary 10.g, p. 66]{M}). In this case $\partial f(x^*)=\lbrace d f(x^*)\rbrace$, where $d f(x^*)$ is the G\^ateaux-derivative of $f$ at $x^*$. For more informations on differentiability of convex functions, we refer to the paper of Moreau \cite{M} and the book of Phelps \cite{Ph}. 
\vskip5mm

This note proves two main results, Theorem \ref{thm1} and Theorem \ref{thm3}. In Theorem \ref{thm1}, we will give a necessary and sufficient condition so that, a convex function possesses a minimum on a given convex subset satisfying some general conditions in infinite dimension : Let $E$ be a topological vector space equipped  with a biorthogonal system $(e_n)_{n\geq 1}$, $(e^*_n)_{n\geq 1}$, where $(e_n)_{n\geq 1}$ is linearly independant. Let $X\subset E$ be a non-empty convex subset of $E$, $x^*\in X$ and $f: X\longrightarrow \R \cup \lbrace +\infty \rbrace$ be a convex function with a non-empty domain. We prove, under general hypothesis on $X$ and $f$, that if $f'(x^* ; e_n)$ exists for all $n\in \N^*$, then we have that $f(x^*)=\inf_{x\in X} f(x)$ if and only if $f'(x^*,e_n)=0, \hspace{2mm} \forall n \in \N^*$. This extends a result known under the G\^ateaux-differentiability assumption. In infinite dimension, the existence of $f'(x^* ; e_n)$ for all $n\in \N^*$ does not implies in general the G\^ateaux-differentiability of $f$ at $x^*$. For example, in the Banach space $l^{\infty}(\N)$, the seminorm $p(x)=\limsup_n |x_n|$ is nowhere G\^ateaux-differentiable but we have that $p'(x, e_n)=0$ for all $n\in \N^*$ and all $x\in l^{\infty}(\N)$ (see example \ref{ex1}). 
\vskip5mm
What seems surprising is that, for certain spaces of infinite dimension, in particular Banach spaces having a Schauder basis, the existence of $f'(x^* ; e_n)$ for all $n\in \N^*$ (where $f$ is convex) is equivalent to the G\^ateaux-differentiability of $f$ at $x^*$ (Theorem \ref{thm3}). More precisely, if $E$ is a Hausdorf locally convex topological vector space equipped  with a biorthogonal system $(e_n)_{n\geq 1}$, $(e^*_n)_{n\geq 1}$, where $(e_n)_{n\geq 1}$ is a topological basis (see Definition \ref{def2}) and if $f : E\longrightarrow \R$ is a convex continuous  function, then $f$ is G\^ateaux-differentiable at $x^*\in E$ if and only if $f' (x^*; e_n)$ exists for all $n \in \N^*$. This result extends, to the infinite dimension, a result that is well known in the finite dimension, namely, in $\R^n$ G\^ateaux (hence Fr\'echet) differentiability of a convex continuous function  at some point is equivalent to the existence of the partial derivatives at this point. An example illustrating this last result is the well know fact about the G\^ateaux-differentiability of the norm $\|.\|_1$ in $l^1(\N)$, which says that $\|.\|_1$ is G\^ateaux-differentiable at $(x_n)_{n\geq 1}\in l^1(\N)$ if and only if $x_n\neq 0$ for all $n\in \N^*$. From our Theorem \ref{thm3}, we can see more simply that $\|.\|_1$ is G\^ateaux-differentiable at $(x_n)_{n\geq 1}\in l^1(\N)$ if and only if it is differentiable in the directions of the basis $(e_n)_{n\geq 1}$ of $l^1(\N)$ which is equivalent to the differentiability of the absolute value $|.|$ at $x_n$ for all $n \in \N^*$, that is, if and only if $x_n\neq 0$ for all $n\in \N^*$. In fact the example of the norm $\|.\|_1$ is a particular case of a more general result given in Proposition \ref{propG}.   
\section{Preliminaries.}
Let $E$ be a topological vector space over the field $\R$ and $E^*$ its topological dual. Let $(e_n)_{n\geq 1}$ be a linearly independant familly of element of $E$ and $(e^*_n)_{n\geq 1}$ be a familly of element of $E^*$. The pair $(e_n)_{n\geq 1}$, $(e^*_n)_{n\geq 1}$ is said to be a biorthogonal system if $\langle e^*_n, e_n \rangle =1$ for all $n\in \N^*$ and  $\langle e^*_n, e_k \rangle =0$ if $n\neq 0$. The linear mappins $P^k : E \longrightarrow E$ are defined for all $k \in \N^*$ as follows

$${\displaystyle x \ {\overset {\textstyle P^k}{\longrightarrow }}\ \ P^k(x)=\sum _{n=0}^{k}\langle e^*_n , x \rangle e_{n}}.$$ 

We define the space $E^k$ as the image of $E$ by $P^k$, that is, $E^k =P^k(E)$, which is a finite dimensional vector space isomorphic to $\R^k$. Let $X$ be a subset of $E$. For all $k \in \N^*$, we denote $X^k:=P^k(X)$ and by $Int_{E^k}(X^k)$ we mean the relative interior of $X^k$, that is the interior of $X^k$ in $E^k\simeq\R^k$. 
\vskip5mm
In all this note, we assume that $E$ is a topological vector space over the field $\R$ equipped with a biorthogonal system $(e_n)_{n\geq 1}$, $(e^*_n)_{n\geq 1}$, where the familly $(e_n)_{n\geq 1}$ is linearly independant. For informations about biorthogonal systems, we reffer to \cite{D}.

\begin{definition} Let $E$ be a topological vector space equipped  with a biorthogonal system $(e_n)_{n\geq 1}$, $(e^*_n)_{n\geq 1}$, where $(e_n)_{n\geq 1}$ is linearly independant. Let $X\subset E$ be a non-empty subset of $E$ and let $x^*\in X$ be a fixed point of $E$. 

$(1)$ (Qualification condition) We say that the set $X$ is qualified at $x^*$ if the following conditions hold.
\begin{itemize}
\item  $P^k(x^*)\in Int_{E^k}(X^k)$ for all $k\in \N^*$

\item  $P^k(X-x^*)\subset X-x^*$ for all $k\in \N^*$.
\end{itemize}

$(2)$ (Pseudo-semicontinuity) Let $f : X\longrightarrow \R \cup \lbrace +\infty \rbrace$ be a function with a non-empty domain. We say that $f$ is pseudo-semicontinuous on $X$ with respect to $x^*$ if for all $x \in X$, $$\limsup_{k\longrightarrow+\infty} f(x^*+P^k(x-x^*))\leq f(x).$$

$(3)$ (Directional-differentiability) We say that $f$ is differentiable at $x^*$ in the directions $(e_n)_{n\geq 1}$ if the following limit exists for all $n\in \N^*$
$${\displaystyle f '(x^* ; e_n):=\lim _{t\to 0 \atop t\neq 0}\,{\frac {1}{t}}{\Big (}f (x^*+t e_n)-f (x^*){\Big )}.}$$
\end{definition}
It is easy to see the following proposition.
\begin{proposition}\label{prop2} The sum of two functions which are pseudo-semicontinuous with respect to some point $a$ of a non-empty subset $X$ of $E$ is also a pseudo-semicontinuous function with respect to $a$.
\end{proposition}
\begin{example} \label{ex1} Let $E= l^{\infty}(\N)$ the Banach space of bounded sequences. Let $p: l^{\infty}(\N) \longrightarrow \R$ be the function defined for all $x=(x_n)\in l^{\infty}(\N)$ by $$p(x)=\limsup |x_n|.$$
Then, 

$(i)$ Clearly, $l^{\infty}(\N)$ equipped with its natural biorthogonal system $(e_n)_{n\geq 1}$, $(e^*_n)_{n\geq 1}$, is qualified at each of its points.

$(ii)$ $p$ is a continuous seminorm ($p(x)\leq \|x\|_{\infty}$ for all $x\in l^{\infty}(\N)$, thus norm continuous), is differentiable in the directions $(e_n)_{n\geq 1}$ at each $x\in l^{\infty}(\N)$ and we have $p '(x ; e_n)=0$ for all $n\in \N^*$ and all $x\in l^{\infty}(\N)$. However, $p$ is  nowhere G\^ateaux-differentiable.

$(iii)$ $p$ is pseudo-semicontinuous on $l^{\infty}(\N)$ with respect to each element $x^*$ satisfying $p(x^*)=0$, but is not pseudo-semicontinuous on $l^{\infty}(\N)$ with respect to $a$ if $p(a)\neq 0$.
\end{example}
\begin{proof}
It is well know that $p$ is a continuous (with respect the norm $\|.\|_{\infty})$ seminorm, but nowhere G\^ateaux-differentiable (see \cite[Example 1.21]{Ph}). We show that $p$ is differentiable at each $x$ in the directions $(e_n)_{n\geq 1}$. Indeed, for each fixed integer $n \in \N^*$ and each $t\in \R$, it is easy to see that $p(x+t e_n)=p(x)$. It follows that $p '(x ; e_n)=0$ for all $n\in \N^*$ and all $x\in l^{\infty}(\N)$. On the other hand, $p$ is pseudo-semicontinuous on $l^{\infty}(\N)$ with respect to each element $x^*$ satisfying $p(x^*)=0$. Indeed, it is easy to see that $p(x^*+P^k(x-x^*))=p(x^*)$ for all $x^*, x\in l^{\infty}(\N)$. So, since $p(x^*)=0$, then we have that $p(x^*+P^k(x-x^*))=0\leq p(x)$ for all $x\in l^{\infty}(\N)$. Thus, $\limsup_{k\longrightarrow+\infty} p(x^*+P^k(x-x^*))\leq p(x)$, for all $x \in X$. If $p(x^*)\neq 0$, then $p(x^*+P^k(0-x^*))=p(x^*)> 0=p(0)$ and so $p$ is not pseudo-semicontinuous on $l^{\infty}(\N)$ with respect to $x^*$ if $p(x^*)\neq 0$.
\end{proof}
\begin{definition} \label{def2} Let $E$ be a topological vector space over the field $\R$ equipped with a biorthogonal system $(e_n)_{n\geq 1}$, $(e^*_n)_{n\geq 1}$. We say that $(e_n)_{n\geq 1}$ is a topological basis of $E$ if for each $x\in E$, there exists a unique sequence $(a_n)$ of real number such that $x=\sum_{n=0}^{+\infty} a_n e_n$, where the convergence is understood with respect to the topology of $E$. In this case we have $a_n = \langle e^*_n , x \rangle$ for all $n\in \N^*$.
\end{definition}
Note that in the Banach space $(l^{\infty}(\N),\|.\|_{\infty})$ the natural basis ($e_n:=(\delta_j^n)$, where $\delta_j^n$ is the Kronecker symbol satisfying $\delta_j^n=1$ if $j=n$ and $0$ if $j\neq n$) is not a topological basis since $\|P^k(x)-x\|_{\infty}$ does not converges to $0$ in general.
\begin{example} We give two classical examples of Hausdorff locally convex topological vector space equipped with topological basis.

$(1)$ Let $\R^{\N}$ be the vector space of all real sequences. We denote $e_n:=(\delta_j^n)$ the elements of $\R^{\N}$ where $\delta_j^n$ is the Kronecker symbol satisfying $\delta_j^n=1$ if $j=n$ and $0$ if $j\neq n$. We equip this space with the distance : for all $x=(x_n)$ and $y=(y_n)$, $$d_{\R^{\N}}(x, y):=\sum_{i=1}^{+\infty} \frac{2^{-i}|x_i-y_i|}{1+|x_i-y_i|}$$
The liear map $e^*_n : x=(x_n)\mapsto x_n$ is a continuous linear functional for each $n \in \N^*$. The space $(\R^{\N},d_{\R^{\N}}(x, y))$ is a Fr\'echet space having a topological basis $(e_n)_{n\geq 1}$. Indeed, for all $x\in \R^{\N}$, $d_{\R^{\N}}(P^k(x),x) \longrightarrow 0$, when $k\longrightarrow +\infty$. Note that the pair $(e_n)_{n\geq 1}$, $(e^*_n)_{n\geq 1}$ is a a biorthogonal system.

$(2)$ Let $(E,\|.\|)$ be a Banach space. A Schauder basis $(e_n)_{n\geq 1}$ is a basis such that for each $x\in E$ there exists a unique sequence of real number $(a_n)$ such that $\|x-\sum_{n=1}^{k} a_n e_n\|\longrightarrow 0$, when $k\longrightarrow +\infty$. The linear  mappings $e^*_k : E\longrightarrow \R$, $k\in \N^*$, are defined by $e^*_k(\sum_{n=0}^{+\infty} a_n e_n)=a_k$. It follows from the Banach-Steinhaus theorem that the linear mappings $P^k : E\longrightarrow E$, $k\in \N^*$, are uniformly bounded by some constant $C$. Also, for all $k\in \N^*$, the linear functionals $e^*_k$ are bounded on $E$. In this case, $(e_n)_{n\geq 1}$, $(e^*_n)_{n\geq 1}$ is a biorthogonal system of $E$.
\end{example}
\begin{proposition} \label{prop1} Let $E$ be a topological vector space equipped with a topological basis $(e_n)_{n\geq 1}$ and a biorthogonal system $(e_n)_{n\geq 1}$, $(e^*_n)_{n\geq 1}$. Then, any uper semicontinuous function (in particular, any continuous function) $f : E \longrightarrow \R$ is pseudo-semicontinuous with respect to each point of $E$. 
\end{proposition}
\begin{proof} Since $(e_n)_{n\geq 1}$ a topological basis, then $P^k(x-x^*)\longrightarrow x-x^*$ for the topology of $E$, when $k\longrightarrow +\infty$ (equivalently $x^*+ P^k(x-x^*)\longrightarrow x$ for the topology of $E$, when $k\longrightarrow +\infty$) and since $f$ is uper semicontinuous, then  $\limsup_k f(x^*+P^k(x-x^*))\leq f(x)$ for all $x\in E$.
\end{proof} 
Note that in Example \ref{ex1}, the function $p$ is norm continuous on $l^{\infty}(\N)$ but not pseudo-semicontinuous with respect $a\in l^{\infty}(\N)$ if $p(a)\neq0$. This is due to the fact that the natual basis of $l^{\infty}(\N)$ is not a topological basis.
\section{The main results.}

We give below, necessary and sufficient condition of optimality. The proof is based on a reduction to the finite dimension. 
\begin{theorem} \label{thm1} Let $E$ be a topological vector space equipped  with a biorthogonal system $(e_n)_{n\geq 1}$, $(e^*_n)_{n\geq 1}$, where $(e_n)_{n\geq 1}$ is linearly independant. Let $X\subset E$ be a non-empty convex subset of $E$ and let $x^*\in X$. Suppose that $X$ is qualified at $x^*$. Let $f: X\longrightarrow \R \cup \lbrace +\infty \rbrace$ be a convex function with a non-empty domain, which is pseudo-semicontinuous with respect to $x^*$ and differentiable at $x^*$ in the directions $(e_n)_{n\geq 1}$. Then, the following assertions are equivalent.

$(a)$ $f(x^*)=\inf_{x\in X} f(x)$

$(b)$  $f'(x^*,e_n)=0, \hspace{2mm} \forall n \in \N^*$
 
\end{theorem}
\begin{proof} The part $(a)\Longrightarrow (b)$ is easy. indeed, suppose that $f(x^*)=\inf_{x\in X} f(x)$. Then, we have that 
$$ 0 \leq f(x)-f(x^*) \hspace{2mm} \forall x\in X.$$
In particular, since $X$ is qualified at $x^*$, for all $n\in \N^*$ there exists $\alpha_n>0$ such that for all $|t| < \alpha_n$, we have that $x^*+te_n\in X$ and so
$$ 0 \leq f(x^* + t e_n)-f(x^*).$$
Thus, we get that $0\leq \lim_{t \longrightarrow 0^+} \frac{f(x^* + t e_n)-f(x^*)}{t} = f'(x^*; e_n)$. Simmilarly, we have $0 \geq \lim_{t \longrightarrow 0^-} \frac{f(x^* + t e_n)-f(x^*)}{t} = f'(x^*; e_n)$. Hence, $f'(x^*,e_n)=0, \hspace{2mm} \forall n \in \N^*$.

Now, we prove $(b)\Longrightarrow (a)$. Let us define $f^k: X^k\subset E^k \longrightarrow \R \cup \lbrace +\infty \rbrace$ as follows: for all $x\in X$, $$f^k(P^k(x)):=f(x^*+P^k(x-x^*)).$$
Then, for all $k\in \N^*$, we have that $f^k$ is Fr\'echet-differentiable at $P^k(x^*)$. Indeed, for all $n \leq k$ we have that $P^k(e_n)=e_n$ and we have that $f^k(P^k(x^*))=f(x^*)$. Thus, since $P^k(x^*)\in Int_{E^k}(X^k)$ for all $k\in \N^*$ (by the qualification condition of $X$ at $x^*$), then for all $n\leq k$ and all small $t$ we have
$$f^k(P^k(x^*)+t e_n)-f^k(P^k(x^*))=f(x^*+t e_n)-f(x^*).$$
It follows that
\begin{eqnarray}\label{eq4bis}
(f^k)'(P^k(x^*); e_n )&=& f'(x^*; e_n).                                         
\end{eqnarray}
This shows that $(f^k)'(P^k(x^*); e_n)$ exists for each $e_n\in E^k$, $n\in \lbrace1,...,k\rbrace$. Since $f$ is convex, then $f^k$ is also convex on the convex set $X^k$. Since $E^k$ is of finite dimension and $((e_n)_{1\leq n \leq k})$ is a basis of $E^k$, then it is well known (see \cite[Theorem 6.1.1.]{KK}) that $f^k$ is Fr\'echet-differentiable at $P^k(x^*)$. 
\vskip5mm
Thanks to the equations $(b)$ and (\ref{eq4bis}), we have that for all $k\in \N^*$,
 \begin{eqnarray*}\label{eq4}
 \nabla f^k(P^k(x^*))=0,
 \end{eqnarray*}
where $\nabla f^k(P^k(x^*))$ denotes the gradiant of $f^k$ at $P^k(x^*)$. It follows that 
\begin{eqnarray}\label{eq6}
 f^k(P^k(x^*)) =\inf_{y\in X^k} f^k(y),
\end{eqnarray}
since $f^k$ is a convex function defined on the convex set $X^k\subset E^k$ and $P^k(x^*)\in Int_{E^k}(X^k)$.

For all $x \in X$ and all $k\in \N^*$, we have that $P^k(x)\in X^k$, then by using (\ref{eq6}) we get
$$f(x^*)= f^k(P^k(x^*)) =\inf_{y\in X^k} f^k(y) \leq f^k(P^k(x)):=f(x^*+P^k(x-x^*)).$$ 
Since $f$ is pseudo-semicontinuous on $X$ with respect $x^*$, then by taking the limit in the above inequality we obtain that
$$f(x^*)\leq \limsup_{k \longrightarrow +\infty} f(x^*+P^k(x-x^*)) \leq f(x).$$
It follows that $f(x^*)=\inf_{x\in X} f(x)$.
\end{proof}
\begin{Rem} We give the following comments about Theorem \ref{thm1}: 

$(1)$ The above Theorem shows that, for a convex function which is pseudo-semicontinuous with respect to $x^*\in E$ and differentiable at $x^*$ in the directions $(e_n)_{n\geq 1}$, a necessary and sufficient condition to have a minimum at $x^*$ is to satisfy $ f'(x^*,e_n)=0, \hspace{2mm} \forall n \in \N^*$. In several examples (see the examples in this note), it is easy to calculate the derivative $f'(x^*,e_n)$ and also to solve $ f'(x^*,e_n)=0, \hspace{2mm} \forall n \in \N^*$. Thus, the candidate for the minimum can be exhibited. Since the condition is also sufficient, we get the point that realizes the minimum. 

$(2)$ In infinite dimension, the differentiability of $f$ in the directions $(e_n)_{n\geq 1}$ at some point $x^*$, does not implies its G\^ateaux-differentiable at $x^*$. An example in the space $l^{\infty}(\N)$ illustrating this situation was given in Example \ref{ex1}. However, in Hausdorf locally convex topological vector spaces equipped  with a biorthogonal system $(e_n)_{n\geq 1}$, $(e^*_n)_{n\geq 1}$, where $(e_n)_{n\geq 1}$ is a topological basis, the situation is different as we show it in Theorem \ref{thm3} below.

$(3)$ If $E$ is a topological vector space equipped with a topological basis $(e_n)_{n\geq 1}$ and a biorthogonal system $(e_n)_{n\geq 1}$, $(e^*_n)_{n\geq 1}$. Then, any uper semicontinuous function (in particular, any continuous function) $f: E \longrightarrow \R$ is pseudo-semicontinuous with respect to each point of $E$ (see Proposition \ref{prop1}). This is not the case for the Banach space $l^{\infty}(\N)$ (see Example \ref{ex1}).

$(4)$ The condition of pseudo-semicontinuity cannot be dropped from the hypothesis of Theorem \ref{thm1} as we will see it in Exemple \ref{ex2}, where $E=l^{\infty}(\N)$.
 
\end{Rem}

We obtain the following immediate corollary. 
\begin{corollary}\label{cor1} Let $E$ be a topological vector space equipped  with a biorthogonal system $(e_n)_{n\geq 1}$, $(e^*_n)_{n\geq 1}$, where $(e_n)_{n\geq 1}$ is linearly independant. Let $x^*\in E$, $p\in E^*$ (the topological dual of $E$) and  $f: E\longrightarrow \R$ be a convex function which is pseudo-semicontinuous with respect to $x^*$ and differentiable at $x^*$ in the directions $(e_n)_{n\geq 1}$. Then, the following assertions are equivalent.

$(a)$ $p\in \partial f(x^*)$

$(b)$ $f'(x^*,e_n)=\langle p, e_n\rangle, \hspace{2mm} \forall n \in \N^*$

\end{corollary}
\begin{proof} We apply Theorem \ref{thm1} with $X=E$ (which is automatically qualified at each of its points) and the convex function $f-p$ which is pseudo-semicontinuous with respect to $x^*$ (see Proposition \ref{prop1} and Proposition \ref{prop2}).
\end{proof}
As consequence, we obtain the following theorem that we announced in the abstract. 

\begin{theorem} \label{thm3} Let $E$ be a Hausdorf locally convex topological vector space equipped  with a biorthogonal system $(e_n)_{n\geq 1}$, $(e^*_n)_{n\geq 1}$, where $(e_n)_{n\geq 1}$ is a topological basis. Let $f : E\longrightarrow \R$ be a convex continuous  function. Then, $f$ is G\^ateaux-differentiable at $x^*\in E$ if and only if $f' (x^*; e_n)$ exists for all $n \in \N^*$. In this case, $d f(x^*) (h)= \sum_{n=1}^{+\infty}f' (x^*; e_n)\langle e^*_n , h\rangle$, for all $h\in E$.
\end{theorem}
\begin{proof} The "only if" part is trivial. Let us prove the "if" part. We know from \cite[Proposition 10.c, p.60, ]{M} that $\partial f(x^*)\neq \emptyset$ for each point $x^*$ of $E$, since $f$ is continuous. We also know from \cite[Corollary 10.g, p. 66]{M}, that $\partial f(x^*)$ is a singleton if and only if $f$ is G\^ateaux-differentiable at $x^*$. On the other hand, using Proposition \ref{prop1}, $f$ is pseudo-semicontinuous on $E$ with respect to each point since it is continuous.  So, let $p, q\in \partial f(x^*)$. Hence, from Corollary \ref{cor1}, we have that $\langle p, e_n\rangle =f'(x^*,e_n)=\langle q, e_n\rangle, \hspace{2mm} \forall n \in \N^*$. It follows that $p=q$ and so that $\partial f(x^*)$ is a singleton, which implies that $f$ is G\^ateaux-differentiable at $x^*$. 
\end{proof}
Note that the result of Theorem \ref{thm3} fails for the Banach space $(l^{\infty}, \|.\|_{\infty})$, see Example \ref{ex1}.
\vskip5mm
It is well know (see for instance \cite[Examples 1.4.]{Ph}) that the norm $\|x\|_1=\sum_{n\geq 1} |x_n|$ in $l^1(\N)$ is G\^ateaux-differentiable at $x=(x_n)$ if and only if $x_n\neq 0$ for all $n\in \N^*$. This fact is a particular case of a more general result given in the following proposition, which is a consequence of Theorem \ref{thm3}. It suffices to take $u_n (t)=|t|$ for all $t\in \R$ and all $n\in \N^*$ in the following proposition.
\begin{proposition} \label{propG} Let $(E,\|.\|)$ be a Banach space having a Schauder basis $(e_n)_{n\geq 1}$ and let $(e_n)_{n\geq 1}$, $(e^*_n)_{n\geq 1}$ be a biorthogonal system. For each $n\in \N$, let $u_n : \R\longrightarrow \R$ be a convex function. Suppose that that $\sum_{n=1}^{+\infty} u_n(\langle e^*_n, x \rangle)$ is a convegrent series for each $x\in E$. 
Let $f : E \longrightarrow \R$ be the function defined by 
$$f(x):= \sum_{n=1}^{+\infty} u_n(\langle e^*_n, x \rangle).$$ 
Then, $f$ is G\^ateaux-differentiable at $x\in E$, if and only if the function $u_n$ is derivable at $\langle e^*_n, x \rangle$, for all $n\in \N^*$. 
\end{proposition}
\begin{proof} First, we have that $f$ is convex and continuous on $E$ since $u_n$ is convex for each $n\in \N^*$ and the domaine of $f$ is $E$. On the other hand, it is clear that for each $n \in \N$, we have that $f'(x,e_n)$ exists if and only if $u_n$ is derivable at $\langle e^*_n, x \rangle$. Thus, using Theorem \ref{thm3}, we obtain the conclusion.

\end{proof} 

\begin{proposition} \label{propGG} Let $(E,\|.\|)$ be a Banach space having a Schauder basis $(e_n)_{n\geq 1}$ and let $(e_n)_{n\geq 1}$, $(e^*_n)_{n\geq 1}$ be a biorthogonal system and let $x^* \in E$ be a fixed points. For each $k\in \N$, let $f_k : E\longrightarrow \R$ be a convex continuous function such that :

$(i)$ for all $n \in \N^*$, there exists $a_n>0$ such that for all $k\in \N^*$, $f_k' (x ; e_n)$ exists for all $x\in I_{x^*,a_n}:=\lbrace x^* +t e_n : |t|<a_n \rbrace$

$(ii)$ the series $\sum_{k=1}^{+\infty} f_k (x)$ converges pointwise to a function $f$.

$(iii)$ for all $n \in \N^*$, the series $\sum_{k=1}^{+\infty} f_k' (x ; e_n)$ converges uniformly with respect to  $x\in I_{x^*,a_n}$.

\vskip5mm

Then, for all $k\in \N^*$, $f_k$ is G\^ateaux-differentiable at $x^*$, the convex continuous function $f$ is G\^ateaux-differentiable at $x^*$ and we have $f'(x^*, e_n)=\sum_{k=1}^{+\infty} f'_k(x^*, e_n)$ for all $n \in \N^*$. 
\end{proposition}
\begin{proof} From part $(i)$ and Theorem \ref{thm3}, we have that the function $f_k$ is G\^ateaux-differentiable at $x^*$, for all $k\in \N^*$. Let us set $\tilde{f}_{k,n,x^*}(t):= f_k(x^*+t e_n)$ and $\tilde{f}_{n,x^*}(t):=f(x^*+te_n)$, for each $k, n \in \N$ and for all $t\in ]-a_n,a_n[$. Using parts $(i)$-$(iii)$, we get that for all $n\in \N^*$, the series $\sum_{k=1}^{+\infty} \tilde{f}'_{k,n,x^*}(t)$ is uniformly convergent on $]-a_n,a_n[$ and the series $\sum_{k=1}^{+\infty} \tilde{f}_{k,n,x^*}(0)$ converges. Thus, for all $n\in \N^*$ the series $\tilde{f}_{n,x^*}(t)=\sum_{k=1}^{+\infty} \tilde{f}_{k,n,x^*}(t)$ is differentiable on $]-a_n,a_n[$ in particular it is differentiable at $t=0$ and we have that $\tilde{f}'_{n,x^*}(0)=\sum_{k=1}^{+\infty} \tilde{f}'_{k,n,x^*}(0)$. In other words, we have that for all $n\in \N^*$, $f'(x^* ; e_n)$ exists and $f'(x^* ; e_n)=\sum_{k=1}^{+\infty} f_k' (x^* ; e_n)$. Hence, $f$ is G\^ateaux-differentiable at $x^*$ by Theorem \ref{thm3}.

\end{proof}

In the following corollary, we give a Karush–Kuhn–Tucke sufficient condition, where G\^ateaux-differentiability is replaced by the weaker condition of derivatives in the directions of $(e_n)_{n\geq 1}$.
\begin{corollary} \label{cor2} Let $E$ be a topological vector space equipped  with a biorthogonal system $(e_n)_{n\geq 1}$, $(e^*_n)_{n\geq 1}$, where $(e_n)_{n\geq 1}$ is linearly independant. Let $X\subset E$ be a non-empty convex subset of $E$. Let $f, g_j, h_k : X\longrightarrow \R$, $j \in \lbrace 1,...,m \rbrace, k\in \lbrace 1,..., p\rbrace$, be convex functions. Let $\Gamma$ the following set
$$\Gamma = \lbrace x\in X/g_j(x)\leq 0, \forall j \in \lbrace 1, ..., m\rbrace; h_k(x)=0, \forall k \in \lbrace 1, ..., p\rbrace\rbrace.$$
Let $x^*\in \Gamma$ and suppose that $X$ is qualified at $x^*$ and that $f, g_j, h_k$ are pseudo-semicontinuous with respect to $x^*$ and differentiable at $x^*$ in the directions $(e_n)_{n\geq1}$. Then, $(1)\Longrightarrow (2)$.
\vskip5mm
$(1)$ There exists $\lambda^*_j \geq 0$ for all $j \in \lbrace 1, ..., m\rbrace$ and $\nu^*_k \in \R$ for all $k \in \lbrace 1, ..., p\rbrace$ such that
\begin{eqnarray*}\label{eq1}
 \lambda^*_jg_j(x^*)=0, \forall j\in \lbrace 1, 2, ...,m \rbrace 
 \end{eqnarray*}
 \begin{eqnarray*}
 f'(x^*,e_n)+\sum_{j=1}^{m} \lambda^*_j g'_j(x^*,e_n) + \sum_{k=1}^{p} \nu^*_k h'_k(x^*,e_n)=0, \hspace{2mm} \forall n \in \N^*
 \end{eqnarray*}
 
$(2)$ We have that 

\begin{eqnarray*}
f(x^*)= \inf_{x\in \Gamma} f(x).
\end{eqnarray*}
\end{corollary}
\begin{proof} $(1)\Longrightarrow (2)$. We apply Theorem \ref{thm1} to the function $f+\sum_{j=1}^{m} \lambda^*_j g_j +\sum_{k=1}^{p} \lambda^*_k h_k$, to get that for all $x\in X$
\begin{eqnarray*}
f(x^*)+\sum_{j=1}^{m} \lambda^*_j g_j (x^*) +\sum_{k=1}^{p} \lambda^*_k h_k(x^*) &\leq& f(x)+\sum_{j=1}^{m} \lambda^*_j g_j (x) +\sum_{k=1}^{p} \lambda^*_k h_k(x). 
\end{eqnarray*}
Since, $\lambda^*_j g_j (x^*)=0$ for all $j\in \lbrace 1, 2, ...,m \rbrace$ by hypothesis and since $h_k(x^*)=0$ for all $k\in \lbrace 1, 2, ...,p \rbrace$ (because $x^*\in \Gamma$), then for all $x \in \Gamma$, we ontain that 
\begin{eqnarray*}
f(x^*)&\leq& f(x)+\sum_{j=1}^{m} \lambda^*_j g_j (x) +\sum_{k=1}^{p} \lambda^*_k h_k(x)\\
      &\leq& f(x).                                      
\end{eqnarray*}
Hence, $f(x^*)= \inf_{x\in \Gamma} f(x)$. 
\end{proof}
\section{examples.}
As proved in Example \ref{ex1}, in infinite dimention, the fact that a convex continuous function $f$ is differentiable at $x^*$ in the directions $(e_n)_{n\geq 1}$ does not implies that $f$ is G\^ateaux-differentiable at $x^*$. We give a simple first examples (Example \ref{ex2} and Example \ref{ex4}) showing how Theorem \ref{thm1} can be applied by using only differentiability in the directions $(e_n)_{n\geq 1}$. In Example \ref{ex3}, we show that the condition of pseudo-semicontinuity cannot be dropped from the hypothesis of Theorem \ref{thm1} or Corollary \ref{cor1}.
\begin{example} \label{ex2} Let $f: (l^{\infty}(\N),\|.\|_{\infty}) \longrightarrow \R$ be the convex continuous function defined by 
$$f(x)=\limsup |x_n| + \sum_{n=1}^{+\infty}\beta^n (x_n^2-\frac{x_n}{n}),$$

where, $0< \beta < 1$ is a fixed real number such that $\sum_{n=1}^{+\infty} \beta^n <2$. The qualification condition is trivial at each point. On the other hand, we have that $f'(x; e_n)=p'(x; e_n)+ \beta^n(2x_n- \frac{x_n}{n})=\beta^n(2x_n- \frac{x_n}{n})$ for all $n\in \N^*$ (We use Example \ref{ex1}). It follows that $f'(x; e_n)=0$ if and only if $x=(\frac 1 {2n})$. In this case $p((\frac 1 {2n}))=0$ and so $p$ is pseudo-semicontinuous with respect to this point (thanks to Example \ref{ex1}). Also $f_1$ is Pseudo continuous with respect this point (easy to verify) and so also $f$ by Proposition \ref{prop2}. Then, we can apply Theorem \ref{thm1}. Hence the sequence $x^*=(\frac 1 {2n})$ is an optimal solution of the problem $\inf_{x\in l^{\infty}(\N)} f(x)$. Note that $f$ is not G\^ateaux-differentiable at $(\frac 1 {2n})$ since $p((x_n))=\limsup |x_n|$ is nowhere G\^ateaux-differentiable (see Example \ref{ex1}) and $f_1((x_n))=\sum_{n=1}^{+\infty}\beta^n (x_n^2-\frac{x_n}{n})$ is G\^ateaux-differentiable at $(\frac 1 {2n})$ (see Proposition \ref{propG}). 
\end{example}
\begin{example} \label{ex4} Let $E= \R^{\N}$ and $X:=l^1(\N)\cap (\R_+)^{\N}$ (convex subset) and let $f: X\longrightarrow \R$ be the convex function defined by $$f((x_n)_n)=\sum_{n=1}^{+\infty} x_n -\sum_{n=1}^{+\infty} 2\beta^n x_n^{\frac{1}{2}}$$ (where $0< \beta < 1$ is a fixed real number). The problem is to minimize $f$ on $X$. A solution of this problem is $x^*=(\beta^{2n})\in X$.
\end{example}
\begin{proof} The function $f$ is  differentiable in the directions $(e_n)_{n\geq 1}$ at each $x=(x_n)\in X$ such that $x_n>0$ for all $n\in \N^*$ and we have $f'(x; e_n)=1- \frac{\beta^n}{(x_n)^{\frac{1}{2}}}$ for all $n\in \N^*$. Now, suppose that $f'(x; e_n)=0$ for all $n\in \N^*$. Then, we have that 
$$1- \frac{\beta^n}{(x_n)^{\frac{1}{2}}}=0, \hspace{2mm}\forall n\in \N^*.$$
In other words, $x_n=\beta^{2n}$ for all $n\in \N^*$. Clearly, the point $x^*=(\beta^{2n})$ belongs to $X$. To schow that $x^*$ is an optimal solution of the problem of minimization, it suffices to prove that $X$ is qualified at $(\beta^{2n})$ and that $f$ is pseudo-semicontinuous on $X$ with respect to $(\beta^{2n})$. This is the case as we will see it. In fact, $X$ is qualified at each point $(x_n)$ such that $x_n>0$ for all $n\in \N^*$ and $f$ is pseudo-semicontiuous on $X$ with respect to each point $x$ of $X$. Indeed, 

$(a)$ for all $k\in \N^*$, we have that $P^k(X)=\R_+^k\times (\lbrace 0 \rbrace)^{\N-k}:=X^k$. It follows that $Int_{E^k}(X^k)=(\R_+^*)^k\times (\lbrace 0 \rbrace)^{\N-k}$. In particular $P^k((\beta^{2n}))\in Int_{E^k}(X^k)$

$(b)$ clearly we have $x^*+P^k(x-x^*)\in X$ for all $k\in \N^*$ and all $x=(x_n), x^*=(x^*_n)\in X$. Indeed, $$x^*+P^k(x-x^*)=(x_1,x_2,...,x_k, x^*_{k+1}, x^*_{k+2},....)\in l^1(\N)\cap (\R_+)^{\N}=X.$$

$(c)$ Let $x, x^*\in X$, then
\begin{eqnarray*}
f(x^*+P^k(x-x^*) -f(x) &=& \sum_{n=k+1}^{\infty} (x^*_n -x_n) - \sum_{n=k+1}^{\infty} 2\beta^n ((x^*_n)^{\frac{1}{2}}- x_n^{\frac{1}{2}})
\end{eqnarray*} 
it follows that $\lim_{k \longrightarrow +\infty} f(x^*+P^k(x-x^*))= f(x)$. Hence, $f$ is pseudo-semicontiuous on $X$ with respect each point $x$ of $X$ in particular with respect the point $x^*=(\beta^{2n})$.
\end{proof}

The following example shows that the condition of pseudo-semicontinuity cannot be dropped from the hypothesis of Theorem \ref{thm1} or Corollary \ref{cor1}.
\begin{example} \label{ex3} Consider the convex continuous function $g: (l^{\infty}(\N),\|.\|_{\infty}) \longrightarrow \R$ defined by 
$$g((x_n)_n)=\limsup |x_n| + \sum_{n=1}^{+\infty} \beta^n (x_n^2-2 x_n),$$
where $0< \beta < 1$ is a fixed real number such that $\sum_{n=1}^{+\infty} \beta^n <2$. We proceed as in Example \ref{ex2} and gives the candidate for the optimum that is $(x_n)=(1,1,1,...)$. However, this candidat is not an optimal solution for the problem $\inf_{x\in l^{\infty}(\N)} g(x)$. Indeed, we verify easly that  $g(1,1,1,...)>g(\frac 1 2,\frac 1 2,\frac 1 2,...)$ (using the fact that $\sum_{n=1}^{+\infty} \beta^n <2$). This is due to the fact that the function $p(x)=\limsup_n (|x_n|)$ is not pseudo-semicontinuus with respect to $(x_n)=(1,1,1,...)$ since we have that $p(1,1,1,...)=1\neq 0$ (see Example \ref{ex1}).  
\end{example}


\bibliographystyle{amsplain}

\end{document}